\documentclass[preprint]{elsarticle}

\usepackage{amssymb,amsmath,amsthm}
\usepackage{cleveref}
\usepackage{tikz}
\usepackage{pgfplots}
\usepackage{pgfplotstable}

\crefname{equation}{}{}
\crefrangeformat{equation}{(#3#1#4) --~(#5#2#6)}
\crefmultiformat{equation}{(#2#1#3)}{, (#2#1#3)}{, (#2#1#3)}{, (#2#1#3)}
\crefname{lem}{lemma}{lemmas}
\Crefname{lem}{Lemma}{Lemmas}

\newcommand{\bx}{{\boldsymbol x}}
\newcommand{\R}{{\mathbb R}}

\newcommand{\C}{{\mathbb C}}

\newcommand{\bmu}{{\boldsymbol \mu}}

\newcommand{\x}{{\boldsymbol x}}
\newcommand{\wcomp}{w_{\text{comp}}}
\newcommand{\eps}{\varepsilon}
\newcommand{\bA}{{\mathbf{A}}}
\newcommand{\y}{{\boldsymbol y}}
\newcommand{\by}{{\boldsymbol y}}
\newcommand{\br}{{\boldsymbol r}}
\newcommand{\bn}{{\boldsymbol n}}
\newcommand{\bl}{{\boldsymbol \ell}}
\newcommand{\bu}{{\boldsymbol u}}
\newcommand{\bs}{{\boldsymbol s}}
\newcommand{\bt}{{\boldsymbol t}}
\newcommand{\bB}{{\mathbf B}}
\newcommand{\bb}{{\boldsymbol b}}
\newcommand{\bc}{{\boldsymbol c}}
\newcommand{\bh}{{\boldsymbol h}}
\newcommand{\bzero}{{\boldsymbol 0}}

\newtheorem{remark}{Remark}

\newtheorem{lem}{Lemma}
\newtheorem{proposition}{Proposition}
\theoremstyle{definition}

\journal{J. Comput. Phys.}
\title{Integral equation formulation of the biharmonic Dirichlet
problem}

\author[yale]{M.~Rachh\corref{cor1}}
\ead{manas.rachh@yale.edu}
\author[uw]{T.~Askham}
\ead{askham@uw.edu}

\cortext[cor1]{Corresponding author}

\address[yale]{Applied Mathematics Program, Yale University,
New Haven, CT 06511}
\address[uw]{Department of Applied Mathematics, University 
of Washington, Seattle, WA 98195-3925}

\begin{document}

\begin{abstract}
We present a novel integral representation for the
biharmonic Dirichlet problem.
To obtain the representation, the Dirichlet
problem is first converted into a related Stokes problem for 
which the Sherman-Lauricella integral representation can
be used. 
Not all potentials for the Dirichlet problem correspond to 
a potential for Stokes flow, and vice-versa, but we show that 
the integral representation can be augmented and modified 
to handle either simply or multiply connected domains.
The resulting integral representation has a kernel which
behaves better on domains with high curvature than 
existing representations. 
Thus, this representation results in more 
robust computational methods for the solution of the
Dirichlet problem of the biharmonic equation and we demonstrate
this with several numerical examples.
\end{abstract}

\begin{keyword}
integral equations \sep biharmonic \sep Dirichlet 
\sep multiply connected
\end{keyword}

\maketitle


\section{Introduction and problem formulation}

A variety of problems of mathematics and physics require the 
computation of a biharmonic potential subject to Dirichlet boundary 
conditions. The pure bending problem for an isotropic and
homogeneous thin clamped plate is a classical application. Another 
application is the computation of a $C^1$ extension of a given 
function from its domain of definition to a larger, enclosing domain
(we discuss these applications further in \cref{sec:apps}). 

The Dirichlet problem is given as follows. For a 
domain $D$ with boundary $\Gamma$, find a function 
$w$ such that

\begin{align}
 \Delta^2 w &= 0 \mbox{ in } D \label{eq:biharmD1} \; ,\\
 w &= f \mbox{ on } \Gamma \label{eq:biharmD2} \; ,\\ 
 \frac{\partial w}{\partial n} &= g \mbox{ on } \partial 
 D, \label{eq:biharmD3}
\end{align}
where $f$ and $g$ are continuous functions defined on 
$\Gamma$. 

The use of standard finite difference methods 
for the solution of \cref{eq:biharmD1,eq:biharmD2,eq:biharmD3} is complicated greatly 
by the fact that the differential equation 
is fourth order. For instance, the
resulting linear system for a discretization with $N$ nodes in 
each dimension
would have a condition number proportional to $N^4$, which  
poses several concerns for obtaining high accuracy solutions for large problems.

Integral equation methods, on the other hand,
have many advantages for such problems. 
Because \cref{eq:biharmD1,eq:biharmD2,eq:biharmD3}
is homogeneous, the resulting integral equation is defined on the 
boundary alone and there is a reduction in the dimension of the
problem. Complex geometries are handled more easily by an 
integral equation and, with appropriate choice of representation, the 
discrete problem tends to be as well conditioned as the underlying 
physical problem, independent of the system size 
\cite{kress1999linear}. 
One challenge 
for integral equation methods
is that the resulting linear systems are dense. However, there
are many well developed fast algorithms for the solution of these
systems, most descending from the fast multipole method (FMM)
\cite{greengard1987}. 

Integral representations for the solution of \cref{eq:biharmD1,eq:biharmD2,eq:biharmD3} have been
developed previously. In particular, the problem is addressed in 
Peter Farkas' thesis \cite{Farkas89} and the method presented there
has been extended to three dimensions in \cite{Jiang11}. The integral
kernels derived in \cite{Farkas89} are taken to be linear 
combinations of derivatives of the fundamental solution of the 
biharmonic problem. Assuming the boundary is a smooth curve, the 
combinations are chosen to maximize the smoothness of the integral
kernel as a function on the boundary (for smooth domains). 
However, the integral kernels derived for \cref{eq:biharmD1,eq:biharmD2,eq:biharmD3} have
a leading order singularity of $r^{-2}$ on a domain with a corner.
Because of this singularity, designing quadrature rules for discretizing
the integral equation is difficult for domains with corners.
Furthermore, the
resulting discretized system has large condition numbers 
for domains whose boundaries have high curvature.

For the related problem of two dimensional steady Stokes flow,
the stream function formulation results in a biharmonic equation
with the gradient of the biharmonic potential specified on
the boundary. Let $w$ be the stream function for Stokes flow 
with no slip boundary conditions, then 
\begin{align}
 \Delta^2 w &= 0 \mbox{ in } D \; , \label{eq:biharmV1}\\
 \frac{\partial w}{\partial \tau} &= f \mbox{ on }
 \Gamma \label{eq:biharmV2} \; , \\ 
 \frac{\partial w}{\partial n} &= g \mbox{ on } 
 \Gamma \, , \label{eq:biharmV3}
\end{align}
for appropriately chosen functions $f$ and $g$. 
Over the past century, much work has been done to develop integral
representations for the biharmonic problem in this setting, as well
as the similar setting of the Airy stress function formulation
of the plane theory of elasticity \cite{greengard1996integral,
power1993completed,power1987second, michlin1957integral}.
The representations given in the above references typically
have more benign singularities than the representation
presented in \cite{Farkas89}. In particular, the representation
used in this paper, taken from \cite{greengard1996integral,
power1993completed}, has a leading 
order singularity of $r^{-1}$ on domains with corners. Moreover, 
this representation (and others from the above references)
can be expressed in terms of  
Goursat functions, allowing for a convenient representation
of the stream function. Because of these advantages, we choose to adapt
the representation of \cite{greengard1996integral,power1993completed} 
to solve \cref{eq:biharmD1,eq:biharmD2,eq:biharmD3}.

This adaptation is not immediate. First, in two dimensional Stokes
flow, the physical quantities of
interest are derivatives of the biharmonic potential $w$ and not $w$
itself; the representation of $w$ from \cite{greengard1996integral,
power1993completed} is not 
necessarily single-valued. Second, in converting the boundary
conditions \cref{eq:biharmD2,eq:biharmD3} into the boundary conditions 
\cref{eq:biharmV2,eq:biharmV3}, the data is differentiated along the curve
so that the original boundary condition is only met up to a constant.
These issues are addressed here, with particular attention
paid to the case of multiply connected domains.
More precisely, we will show that the desired (and uniquely defined)
potential can be expressed in terms of (possibly) multi-valued
Goursat functions.

The rest of the paper is organized as follows. In \cref{sec:prelim}, we  
present some mathematical preliminaries, including the notation used
throughout the paper, a review of the Farkas integral representation,
and a review of the completed layer potential representation for solving 
\cref{eq:biharmV1,eq:biharmV2,eq:biharmV3} in terms of the Goursat
functions. In \cref{sec:anapp}, we explain how to adapt the Stokes
layer potentials for the Dirichlet problem, present an integral
representation for solving \cref{eq:biharmD1,eq:biharmD2,eq:biharmD3},
and prove the invertibility of the resulting integral
equation. We outline the numerical tools we used to solve
this integral equation and  present 
some numerical results in \cref{sec:results}. In \cref{sec:conclusion}, 
we provide some concluding remarks and ideas for future
research.

\section{Preliminaries} \label{sec:prelim}
The notation for the following concepts can be cumbersome and 
an attempt has been made to stay consistent. Vector-valued
quantities are denoted by bold, lower-case letters
(e.g. $\bh$), while tensor-valued quantities are bold
and upper-case (e.g. $\mathbf{T}$). Subscript indices of the
non-bold character (e.g. $h_i$ or $T_{ijk}$)
are used to denote the entries within a vector or tensor.
We use the standard Einstein summation convention, i.e., 
there is an implied sum taken over the repeated indices of 
any term. The vectors $\bx$ and
$\by$ are reserved for spatial variables in $\R^2$, while
$z$ and $\xi$ are reserved for spatial variables in $\C$. 
We make the standard identification between points
in $\R^2$ and points in $\C$, i.e. the point 
$\bx = (x_1,x_2)^\intercal \in \R^2$ is equivalent to
the point $z = x_1+i x_2$, and we switch between the
two notions implicitly in much of what follows.
For integration, the symbol $dS$ is used to denote an 
integral with respect to arc length and the symbol $d\xi$ is
used to denote a complex contour integral. Script letters
$\mathcal{X}$, $\mathcal{Y}$, and $\mathcal{Z}$ are reserved for Banach spaces.
$I_{\mathcal{X}}: \mathcal{X}\to\mathcal{X}$
denotes the identity operator on $\mathcal{X}$.

Let $D$ will denote a bounded, possibly multiply-connected,
domain in $\R^2$ with a smooth boundary $\Gamma$ (unless otherwise
noted). For a domain with $N$
holes, we will denote the outer boundary by $\Gamma_0$ and the
boundary of each hole by $\Gamma_1, \ldots, \Gamma_N$, so that
$\Gamma = \cup_{i=0}^N \Gamma_i$. Let $\bn(\bx)$ denote the outward
unit normal and $\boldsymbol{\tau}(\bx)$ the positively-oriented
unit tangent for $\bx \in \Gamma$. 
If we need to distinguish between the exterior
and interior of $\Gamma$, we will let $D^- = D$ denote the
interior and $D^+ = \R^2 \setminus (D \cup \Gamma)$ denote
the exterior.

\subsection{Applications of the biharmonic Dirichlet problem}
\label{sec:apps}

Consider the pure bending of an isotropic and homogeneous thin 
clamped plate. In the Kirchoff-Love theory, the vertical displacement 
of the plate, $w$, satisfies the equations
\begin{align}
 -\Delta^2 w =  q &\quad \bx \in D \\
 w = 0 &\quad \bx \in \Gamma \\ 
 \frac{\partial w}{\partial n} = 0 &\quad \bx \in \Gamma, 
\end{align}
where $D \subset \mathbb{R}^2$ represents the midline of the thin 
plate, $\Gamma$ is its boundary, and $q$ is the transverse load 
applied to the plate. Using standard techniques, the above problem 
can be reduced to a homogeneous biharmonic problem of the form
\cref{eq:biharmD1,eq:biharmD2,eq:biharmD3}.

In a recent paper, \cite{Askham2017}, it was shown that the 
solution of polyharmonic
Dirichlet problems can be used 
as part of the solution of \textit{inhomogeneous} PDEs 
on complex geometries. We will breifly review this procedure
here. 

Consider the Poisson equation

\begin{align}
  \Delta u &= f &\bx \in D \; \label{eq:poiss1}, \\
  u &= g &\bx \in \Gamma \; .
\end{align}
A particular solution, $v$, which satisfies \cref{eq:poiss1}
can be obtained from the formula 

\begin{equation}
  v(\x) = -\dfrac1{2\pi} \int_\Omega \log |\x-\y| 
  \tilde{f}(\y) \, dy \; , \label{eq:volint}
\end{equation}
where $\Omega$ is some domain such that $D\subset \Omega$ and
$\tilde{f}$ is a function defined on $\Omega$ which 
satisfies $\tilde{f}|_D = f$. 

There are rapid methods for evaluating the integral 
\cref{eq:volint} in the case that $\Omega$ is a box.
However, it is unclear how best to define (and compute)
the values of $\tilde{f}$ on $\Omega \setminus D$, particularly
such that $\tilde{f}$ is smooth across the boundary of
$D$. One approach is to
compute the extension as the solution of a homogeneous
PDE on the exterior. 

Suppose that $w$ solves

\begin{align}
 \Delta^2 w = 0 &\quad \bx \in \Omega \setminus D 
\nonumber\\
 w = f &\quad \bx \in \Gamma \nonumber \\ 
 \frac{\partial w}{\partial n} =\frac{\partial f}{\partial n}  
&\quad \bx \in \Gamma  \label{eq:biharm2}\\
 w = 0 &\quad \bx \in \partial \Omega \nonumber\\
 \frac{\partial w}{\partial n} = 0 &\quad \bx \in \partial 
\Omega \nonumber \; ,
\end{align}
which is a problem of the form 
\cref{eq:biharmD1,eq:biharmD2,eq:biharmD3}.
Then, setting $\tilde{f}|_D = f$ and $\tilde{f}|_{\Omega \setminus D} = w$ 
makes $\tilde{f}$ a $C^1$ function across $\Gamma$. 

In \cite{Askham2017}, a $C^0$ extension was computed
as the solution of a Laplace problem on $\Omega \setminus D$.
This was found to accelerate the convergence of the
Poisson solver over discontinuous extension (i.e. $\tilde{f}$ 
to be zero outside of $D$).
By computing a smoother extension, as in the solution of the
problem above, the efficiency and robustness
of the Poisson solver could be further improved.
For a PDE-based version of this approach, see 
\cite{stein2016}.

\subsection{The Farkas integral representation} 

As mentioned in the introduction, there are existing 
integral representations for the solution of 
\cref{eq:biharmD1,eq:biharmD2,eq:biharmD3}. In \cite{Farkas89}, 
the solution is given as the sum of two layer potentials, i.e.

\begin{equation}
  w(\x) = \int_\Gamma K_1^F (\x,\y) \sigma_1(\y) \, dS(\y)
  + \int_\Gamma K_2^F(\x,\y) \sigma_2(\y) \, dS(\y) \; ,
  \label{eq:farkRep}
\end{equation}
where $\sigma_1$ and $\sigma_2$ are unknown densities.

The integral kernels, $K_1^F$ and $K_2^F$ are based on
derivatives of the Green's function for the biharmonic equation.
For two points on the plane, $\x$ and $\y$, the Green's function
is given by

\begin{equation}
  G^B(\x,\y) = \dfrac1{8\pi} |\x-\y|^2 \log |\x-\y| \; .
\end{equation}
Let $\br = \y - \x$ and $r = |\y-\x|$. Then,

\begin{align}
  K_1^F(\x,\y) &= G^B_{n_yn_yn_y}(\x,\y)
  + 3 G^B_{n_y\tau_y\tau_y}(\x,\y) \; , \label{eq:kf1} \\
  K_2^F(\x,\y) &= -G^B_{n_yn_y}(\x,\y)
  + G^B_{\tau_y\tau_y}(\x,\y) \label{eq:kf2} \; .
\end{align}
More explicitly, we have

\begin{align}
  K_1^F(\x,\y) &= \dfrac1{\pi}
  \dfrac{(\br\cdot \bn (\y))^3}{(\br\cdot \br)^2} \; ,\\
  K_2^F(\x,\y) &= \dfrac1{2\pi} \left ( \dfrac12 - 
  \dfrac{(\br\cdot \bn (\y))^2}{\br\cdot \br} \right ) \; .
\end{align}

On enforcing the Dirichlet boundary conditions for
$w$, we obtain the integral equation

\begin{equation}
  \begin{split}
    \begin{pmatrix}
      f(\x) \\
      g(\x)
    \end{pmatrix} &= 
    \int_\Gamma
    \begin{pmatrix}
      K_{11}^F(\x,\y) & K_{12}^F(\x,\y) \\
      K_{21}^F(\x,\y) & K_{22}^F(\x,\y)
    \end{pmatrix}
    \begin{pmatrix}
      \sigma_1(\y)\\
      \sigma_2(\y)
    \end{pmatrix}
    \, dS(\y) \\
    &\quad + \begin{pmatrix}
      1/2 & 0 \\
      -\kappa(\x) & 1/2
    \end{pmatrix}
    \begin{pmatrix}
      \sigma_1(\x)\\
      \sigma_2(\x)
    \end{pmatrix} \; , \label{eq:farkIntEq} 
  \end{split}
\end{equation}
where $\kappa$ denotes the signed curvature as
a function on $\Gamma$ and $\x$ is a point on $\Gamma$. 
The kernels are given
by $K_{11}^F = K_1^F$, $K_{12}^F = K_2^F$,
\begin{align}
  K_{21}^F(\x,\y) &= \left(K_1^F(\x,\y)\right )_{n_x} \nonumber \\
  &= \dfrac1{\pi} \left ( -3 \dfrac{ (\br \cdot \bn(\y))^2
    (\bn(\x) \cdot \bn(\y))}{(\br \cdot \br)^2}
  +4 \dfrac{ (\br \cdot \bn(\y))^3 (\br \cdot \bn(\x))}
  { (\br \cdot \br)^3} \right )  \; \label{eq:farkK21},\\
  K_{22}^F(\x,\y) &= \left(K_2^F(\x,\y)\right )_{n_x} \nonumber \\
  &= \dfrac1{\pi} \left ( \dfrac{ (\br \cdot \bn(\y)) (\br \cdot \bn(\x))}
  {\br\cdot\br} - \dfrac{(\br \cdot \bn(\y))^2 (\br\cdot \bn(\x))}
  {(\br\cdot\br)^2} \right ) \; .
\end{align}
For a sufficiently smooth and simply connected domain,
the integral equation \cref{eq:farkIntEq} is invertible.
The case of a multiply connected domain is not treated
fully in \cite{Farkas89}, but some of the issues are
considered.

As mentioned above, the kernels $K^F_1$ and $K^F_2$ are
constructed with the goal that the $K^F_{ij}$ are
as smooth as possible. Suppose that the boundary
$\Gamma$ is of class $C^k$. Then, the kernels,
$K^F_{ij}(\x,\y)$, are $C^{k-2}$ functions on the boundary
for each $\y \in \Gamma$ \cite{Farkas89}.
Therefore, on a smooth
boundary, these kernels are smooth. However, on a
domain with a corner, it is clear from the
formula \cref{eq:farkK21} that the kernel $K_{21}^F$
has a singularity with strength $r^{-2}$. This singularity, 
in addition to the term in \cref{eq:farkIntEq}
which explicitly involves the
curvature, makes the representation \cref{eq:farkRep}
unstable for domains with high curvature (or
corners).

\subsection{Stokes flow in the plane} 
\label{subsec:stokesflow}

The equations of incompressible Stokes flow
with no-slip boundary conditions on a domain 
$D$ with boundary $\Gamma$ are
\begin{align}
-\Delta\bu+\nabla p & =0\quad\mbox{in D},
\label{eq:StokesFlowEq} \\
\nabla\cdot\bu & =0\quad\mbox{in D},
\label{eq:MassConservation}\\
\bu &= \bh \quad\text{on }\Gamma, 
\label{eq:bcStokesFlow}
\end{align}
where $\bu$ is the velocity of the fluid and $p$
is the pressure. Following
standard practice, the velocity $\bu$ can be 
represented by a stream function $w$. Let
\begin{equation}
\bu = \nabla^{\perp} w = \left( 
\begin{array}{r} \frac{\partial w}{\partial x_2} \\ 
-\frac{\partial w}{\partial x_1} \end{array} \right) \label{eq:StreamFunc},
\end{equation}
so that the divergence-free condition, 
\cref{eq:MassConservation}, is satisfied automatically. 
Taking the dot product of $\nabla^\perp$ with 
\cref{eq:StokesFlowEq} results in a biharmonic equation
for $w$. In particular, $w$ is a biharmonic function 
which satisfies \cref{eq:biharmV1,eq:biharmV2,eq:biharmV3}
with $f = -h_i n_i$ and $g = h_i \tau_i$.

\subsection{Goursat functions}

Goursat showed that any biharmonic function $w$ 
can be represented by two analytic functions $\phi$ and $\psi$ 
(called Goursat functions) as
\begin{equation}
 w \left(x_1,x_2\right) = \text{Re} \left(\bar{z} \phi \left(z \right) 
+ \chi \left(z\right) \right) , \label{eq:wGoursat}
\end{equation}
where $\chi^{'} = \psi$ and $z=x_1+i x_2$ \cite{goursat1898equation}. 
In solving equation 
\cref{eq:biharmV1,eq:biharmV2,eq:biharmV3}, we are interested in 
$\frac{\partial w}{\partial x_1}$ and $\frac{\partial w}{\partial x_2}$. 
Muskhelishvili's formula \cite{muskhelishvili1977some} gives an expression
for these quantites in terms of the Goursat functions as
\begin{equation}
 \frac{\partial w}{\partial x_1}+ i \frac{\partial w}{\partial x_2} = 
\phi\left(z\right) + z\overline{\phi^{'} \left(z\right)} + 
\overline{\psi \left(z\right)} \, . \label{eq:MushkiFormula}
\end{equation}
We say that a pair of Goursat functions $\phi$ and $\psi$
is \textit{equivalent} to a 
Stokes velocity field $\bu$ if the biharmonic function $w$ 
defined by \cref{eq:wGoursat} is such that $\bu = 
\nabla^\perp w$.

The references \cite{greengard1996integral,power1993completed,
power1987second, michlin1957integral} give many options for the 
representation
of $\phi$ and $\psi$ as layer potentials of a complex 
density given on the boundary of the domain. 
Of computational interest, representations
for $\phi$ and $\psi$ exist
such that enforcing the boundary conditions of \cref{eq:biharmV1,eq:biharmV2,eq:biharmV3}
results in an invertible second-kind integral equation (SKIE) 
for the density.

\subsection{Integral representations for Stokes flow in the plane}

We will first present
the single layer and double layer potentials of Stokes flow
in the Stokeslet/stresslet formulation, which may be more
familiar. For details concerning these ideas, see, inter alia, 
\cite{pozrikidis1992boundary}.
We will then present their equivalent potentials
in the classical Goursat function formulation. The 
reason for doing so is two-fold: first, the Goursat
formulation makes it more natural to evaluate the stream
function $w$; second, the complex variables-based Goursat
formulation is readily adaptable for efficient fast multipole
methods.

\subsubsection{Stokes layer potentials} \label{sec:stokeslayer}

The Green's function $\mathbf{G}$ for the incompressible Stokes 
equations in free space, or \textit{Stokeslet}, is given by 
\begin{equation}
G_{ij}\left(\bx,\by\right)=\frac{1}{4\pi}\left[ 
-\log\left|\bx-\by\right| \delta_{ij} + 
\frac{\left( x_{i} - y_{i} \right) \left( x_{j} - y_{j}
 \right)} {\left| \bx - \by \right|^2}\right] \quad 
i,j \in \{1,2\} \, .
\end{equation}
The vector field $u_i = G_{ij}\left(\bx,\by
\right) f_j$ represents a Stokes velocity field 
at $\bx$ due to a point force 
$\mathbf{f}$ applied at $\by$.
For a continuous distribution of surface forces $\boldsymbol\mu$ on
a curve $\Gamma$, the induced Stokes field, called a single layer 
potential, is given by 
\begin{equation}
\label{stokesslpdef}
\left [ {S}_{\Gamma}  \mu \right ]_i 
\left(\bx\right)=
\int_{\Gamma} G_{ij} \left(\bx,\by\right)
\mu_{j}\left(\by\right) \, dS(\by) \quad i=1,2 \, . 
\end{equation}

The following lemma describes the behavior of the Stokes single
layer potential as a function on $\R^2$, see 
\cite{pozrikidis1992boundary} for details.
\begin{lem}
\label{lem:slppropspart1}
Let 
$\mathbf{S}_{\Gamma}\mathbf{\bmu}(\bx)$ denote a single layer Stokes potential 
of the form \cref{stokesslpdef}. Then,
$\mathbf{S}_{\Gamma}\bmu\left(\bx\right)$ satisfies 
the Stokes equations
in $\R^{2}\backslash\Gamma$ and
$\mathbf{S}_{\Gamma}\bmu(\bx)$ is continuous
in $\R^{2}$.
\end{lem}

The Stokes single layer potential has equivalent Goursat 
functions, $\phi_{S}$ and $\psi_{S}$, which can be expressed
in terms of complex layer potentials:
\begin{align}
\phi_{S}(z) & =-\frac{1}{8\pi}\int_{\Gamma}\rho\left(\xi\right)
\log\left(\xi-z\right)\, dS(\xi) +
                      \frac{1}{8\pi} \int_{\Gamma} \rho\left(\xi\right) 
                      \, dS(\xi) \, , \label{eq:SLGoursatPhi}\\
\chi_{S}(z) & =\frac{1}{8\pi}\int_{\Gamma}
\overline{\rho\left(\xi\right)}\left(\xi-z\right)
\left[\log\left(\xi-z\right)-1\right] \, dS(\xi) \nonumber \\ 
&\quad +\frac{1}{8\pi}\int_{\Gamma}\overline{\xi}\rho(\xi)\log
\left(\xi-z\right) \, dS(\xi) \, ,\label{eq:SLGoursatChi}\\
\psi_{S}(z) & =-\frac{1}{8\pi}\int_{\Gamma}
\overline{\rho\left(\xi\right)}\log\left(\xi-z\right) \, dS(\xi)-
\frac{1}{8\pi}\int_{\Gamma}\frac{\overline{\xi}\rho\left(\xi\right)}
     {\xi-z}\, dS(\xi) \, , \label{eq:SLGoursatPsi}
\end{align}
where $z=x_1 + ix_2$, $\xi = y_1 + iy_2$, and $\rho = \mu_2 - i \mu_1$. 
The stream function $w_S$ corresponding to this 
Goursat function pair is then
\begin{equation}
\begin{split}
 w_{S}(z)  &=\mbox{Re}\left[\frac{1}{4\pi}\int_{\Gamma}\mbox{Re}
\left[\overline{\left(\xi-z\right)}\rho\left(\xi\right)\right]
\log\left(\xi-z\right) \, dS(\xi) \right. \\ 
&\quad \left. -\frac{1}{8\pi}\int_{\Gamma}\overline{\rho
\left(\xi\right)}\left(\xi-z\right)\, dS(\xi) + \overline{z}\frac{1}{8\pi} 
\int_{\Gamma} 
\rho\left(\xi\right) \, dS(\xi) \right] \\
 & =: S^w_\Gamma \rho \, .  \label{eq:SLStreamFunc}
\end{split}
\end{equation}
Note that, the velocity field associated with the stream function 
$w_{S}$
is given by
\begin{equation}
\nabla^{\perp} w_{S} \left(z\right)= 
\nabla^{\perp} S^w_\Gamma \rho \left(z\right)
= \mathbf{S}_{\Gamma} \boldsymbol{\mu} \left(\bx\right) \, .
\label{eq:SLStreamFuncEquiv}
\end{equation}

Another quantity of physical interest in Stokes flow is the 
stress tensor $\boldsymbol\sigma$; for a Stokes velocity 
field $\bu$ and pressure $p$, it is given by
\begin{equation}
 \sigma_{ij} =-p\delta_{ij}+\left(\frac{\partial u_{i}}{\partial x_{j}}+
\frac{\partial u_{j}}{\partial x_{i}}\right) \, .
\end{equation}
The stress tensor $\mathbf{T}$, or \textit{stresslet}, 
associated with the 
Green's function $\mathbf{G}$ is given by
\begin{equation}
T_{ijk} \left(\bx,\by\right) = -\frac{1}{\pi} 
\frac{\left( x_{i} - y_{i} \right) \left( x_{j} - y_{j} \right) 
\left( x_{k} - y_{k} \right)}{\left| \bx - \by \right|^{4}}
\, .
\end{equation} 
The vector field $u_i = T_{ijk}(\by,\bx) n_k f_j$ 
represents the velocity field resulting from a stresslet with strength
$\bf$ oriented in the direction
$\bn$ at $\by$.
For a distribution of stresslets $\boldsymbol\mu$ on
a curve $\Gamma$, the induced Stokes field, called a double 
layer potential, is given by
\begin{equation}
\label{stokesdlpdef}
\left [ {D}_{\Gamma} {\mu} \right ]_i 
\left(\bx\right) =
\int_{\Gamma}
T_{ijk}\left(\by,\bx \right) n_k  
\left(\by\right) 
\mathbf{\mu}_j \left(\by\right) \, dS(\by)
\, .
\end{equation}

The following lemma describes the behavior of the Stokes double
layer potential as a function on $\R^2$, see 
\cite{pozrikidis1992boundary} for details.
\begin{lem} \label{lem:dlppropspart1}
Let $\Gamma$ be a curve, $D^+$ denote the exterior of the
curve, $D^-$ denote the interior, $\mathbf{D}_{\Gamma}\bmu
\left(\bx\right)$ denote a double layer Stokes
potential of the form \cref{stokesdlpdef}, and
$\x_0 \in \Gamma$. Then,
$\mathbf{D}_{\Gamma}{\bmu}\left(\bx\right)$ satisfies the
Stokes equations in $\R^{2}\backslash\Gamma$ and the jump relations:
\begin{align}
\lim_{\substack{\bx\to\bx_{0}\\
\bx\in D^{\pm}}} 
\left [ {D}_{\Gamma} {\mu}\right ]_{i} (\x) &=
\pm\frac{1}{2}\mathbf{\mu}_{i}
\left(\bx_{0}\right)+\oint_{\Gamma}{T}_{j,i,k} \left(\by, 
\bx_{0}\right) n_{k}\left(\by\right) {\mu_j}
\left(\by\right) \, dS(\y)  \\
&=: \pm\frac{1}{2}\mathbf{\mu}_{i}(\x_0) 
+ \left [ {D}_{\Gamma}^{PV}
{\mu} \right ]_{i} (\x_0) \, .
\end{align}
In the above, $\oint$ denotes a Cauchy principal value
integral and $\mathbf{D}_\Gamma^{PV} \bmu$ denotes the double layer
potential with the integral taken in the principal value sense.

\end{lem}

The double layer potential above has equivalent Goursat functions,
$\phi_D$ and $\psi_D$, which can be expressed in
terms of complex layer potentials:
\begin{align}
\phi_D(z) & =-\frac{1}{4\pi i}\int_{\Gamma}\frac{\rho
\left(\xi\right)}{\xi-z}d\xi \label{eq:DLGoursatPhi} \, ,\\
\chi_D(z) & =\frac{1}{4\pi i}\int_{\Gamma}\left(\overline{\rho
\left(\xi\right)}d\xi+\rho\left(\xi\right)\overline{d\xi}\right)\log
\left(\xi-z\right)
+\frac{1}{4\pi i}\int_{\Gamma}\frac{\overline{\xi}\rho
\left(\xi\right) \, d\xi}{\xi-z} \, , \label{eq:DLGoursatChi} \\
\psi_D(z) & =-\frac{1}{4\pi i}\int_{\Gamma}\frac{\overline{\rho
\left(\xi\right)}d\xi+\rho\left(\xi\right)\overline{d\xi}}{\xi-z}
+\frac{1}{4\pi i}\int_{\Gamma}\frac{\overline{\xi}\rho\left(\xi\right) \, d\xi}
{\left(\xi-z\right)^{2}} \, , \label{eq:DLGoursatPsi}
\end{align}
where $z, \xi$, and  $\rho$ are as above. The stream function 
$w_{D}$ corresponding to this Goursat function pair is
then
\begin{equation}
\begin{split}
 w_D(z) &=\mbox{Re}\left[\frac{1}{4\pi i}\int_{\Gamma}
\frac{\overline{\xi-z}}{\xi-z}\rho\left(\xi\right)d\xi
+\frac{1}{4\pi i}\int_{\Gamma}\left(\overline{\rho\left(\xi\right)}d\xi
+\rho\left(\xi\right)\overline{d\xi}\right)\log\left(\xi-z\right)\right] 
\\
&=: D^w_\Gamma \rho\label{eq:wDL} \, .
\end{split}
\end{equation}
As before, the velocity field associated with the stream function
$w_D$ is given by
\begin{equation}
\nabla^{\perp} w_D\left(z\right) = 
\nabla^{\perp} D^w_\Gamma \rho \left(z\right)=  
\mathbf{D}_{\Gamma} \boldsymbol{\mu} \left(\bx\right) \, .
\label{eq:DLStreamFuncEquiv}
\end{equation}

\subsubsection{The completed double layer representation for Stokes
flow}
Using the layer potentials described above, we can represent
the solution of $\cref{eq:biharmV1,eq:biharmV2,eq:biharmV3}$, or equivalently the system
\cref{eq:StokesFlowEq}, \cref{eq:MassConservation}, and
\cref{eq:bcStokesFlow}, in terms of a density $\boldsymbol\mu$
given on the boundary of the domain. The completed double
layer representation \cite{power1993completed} for the velocity is
\begin{equation}
 \bu(\x) = \mathbf{S}_{\Gamma} {\boldsymbol \mu} (\x) 
+ \mathbf{D}_{\Gamma} {\boldsymbol \mu} (\x) +  
\mathbf{W}_\Gamma {\boldsymbol \mu}  \label{eq:IntRepStokes} \; ,
\end{equation}
where $\mathbf{W}_\Gamma \boldsymbol{\mu} = 
\int_\Gamma \boldsymbol{\mu} \, dS$,
and the representation of an equivalent pair of Goursat
functions, giving a stream function $w$,
can be inferred from the formulas of the previous subsection.
When the no-slip boundary conditions are enforced for this
representation, the result is an invertible SKIE for the 
density $\boldsymbol\mu$. The reader can refer to \cite{power1993completed}
for a detailed discussion of the Fredholm alternative for
this representation. We summarize it as

\begin{lem} \label{lem:stokesrep}
Let $\bu$ be defined as in \cref{eq:IntRepStokes}
and $\x_0 \in \Gamma$. Then 

\begin{equation}
\begin{split}
\lim_{\substack{\bx\to\bx_{0}\\
\bx\in D^{\pm}}} \bu (\x) &=
\pm\frac{1}{2}\boldsymbol{\mu} (\x_0) 
+ \mathbf{S}_\Gamma \boldsymbol{\mu} (\x_0)
+ \mathbf{D}^{PV}_\Gamma \boldsymbol{\mu} (\x_0)
+ \mathbf{W}_\Gamma \boldsymbol{\mu} \\
& =: \pm\frac{1}{2}\boldsymbol{\mu}(\x_0) 
+ \mathbf{K}_\Gamma \boldsymbol{\mu} (\x_0) \, .
\end{split}
\end{equation}
For a sufficiently smooth curve $\Gamma$, the
operator $\mathbf{K}_\Gamma$ is a compact operator on
$\mathcal{X} \times \mathcal{X}$, where 
$\mathcal{X}$ is $L^2(\Gamma)$ or $C^{0,\alpha}(\Gamma)$ for 
$\alpha \in (0,1)$. Further, the integral equation

\begin{equation}
  \left (- \dfrac{1}{2} \mathbf{I}_{\mathcal{X}\times\mathcal{X}}
+ \mathbf{K}_\Gamma \right ) \boldsymbol{\mu} = \bh \;
\end{equation}
is invertible, even for multiply connected 
domains.
\end{lem}

For the above integral equation, the singularities 
of the integral kernels which define $\mathbf{K}_\Gamma$ 
are at worst order $r^{-1}$, even for a boundary 
with a corner.

\section{Integral equation derivation} \label{sec:anapp}

We would like to adapt the completed double layer representation
for solutions of Stokes flow \cref{eq:biharmV1,eq:biharmV2,eq:biharmV3} to solve 
the clamped plate problem \cref{eq:biharmD1,eq:biharmD2,eq:biharmD3}.
Let $f$ and $g$ be the boundary data as in \cref{eq:biharmD1,eq:biharmD2,eq:biharmD3}. 
By computing tangential derivatives of $f$ on each
boundary component, we get the following related Stokes problem:
\begin{align}
 \Delta^2 \tilde{w} = 0 &\quad \bx \in D \, ,\nonumber\\
 \frac{\partial \tilde{w}}{\partial \tau} = 
 \frac{\partial f}{\partial \tau} 
&\quad \bx \in \Gamma \label{eq:biharm5} \, ,\\ 
 \frac{\partial \tilde{w}}{\partial n} = g &\quad \bx 
\in \Gamma 
\, . 
\nonumber
\end{align}
There are two main issues to be addressed in using the completed
double layer representation in this context. First,
the representation is designed for 
Stokes flow, in which the quantities of interest are derivatives
of the potential $\tilde{w}$ and not $\tilde{w}$ itself; the
representation for $\tilde{w}$ may not be single-valued. 
We will establish that, in the context of \cref{eq:biharm5},
the stream function is necessarily single-valued. We also discuss
some numerical issues related to evaluating the stream function.
The second issue to address
is that the solution $\tilde{w}$ only satisfies the original boundary 
condition for the value of $\tilde{w}$ up to a constant on 
each boundary component.
In fact, for multiply connected domains, the completed double layer 
representation is incomplete for the Dirichlet
problem \cref{eq:biharmD1,eq:biharmD2,eq:biharmD3}. We present a remedy for this 
issue and provide some physical intuition.

\subsection{Single-valued stream functions}
To solve the Dirichlet problem \cref{eq:biharmD1,eq:biharmD2,eq:biharmD3},
it is necessary to compute
a single-valued biharmonic potential.
In the case of a multiply connected domain,
there is no guarantee that a single-valued stream 
function exists for a given velocity field.

Consider the following example.
Let $(r,\theta)$ denote standard polar coordinates. It is
easy to verify that the velocity field 
$\bu=\frac{1}{r}\hat{r}$ 
solves the equations of Stokes flow in an annulus centered at the 
origin. A stream function for this flow is $w=\theta$, which 
is not single-valued; indeed, there are no single-valued
stream functions which generate this flow. 

Let $D$ be a multiply connected domain with boundary 
$\Gamma = \cup_{i=0}^N \Gamma_i$, as in the previous section.
We note that the gradient of a stream function is determined 
by the velocity field, i.e. 

\begin{equation}
\nabla w = -\bu^\perp := \begin{pmatrix} 
  - u_2 \\ u_1 
\end{pmatrix} \; .
\end{equation}
Therefore, a velocity field has single-valued stream
functions if and only if $\bu^\perp$ is conservative.
Using standard results from multivariable calculus, 
we can  characterize such flows.

\begin{proposition} 
Suppose that $\bu$ is a divergence-free velocity 
field which is $C^1$ on $D$ and continuous on 
$D\cup \Gamma$. The field $\bu^\perp$ 
is conservative if and only if 

\begin{equation}
 \int_{\Gamma_i} \bu\cdot \bn \, dS 
= 0 \quad i=0,1,\ldots N \, .
\label{eq:NecSuffStreamFuncExistence}
\end{equation}

\end{proposition}

The equalities \cref{eq:NecSuffStreamFuncExistence}
constitute $N$ linearly independent constraints on the 
boundary data because the divergence-free condition 
\cref{eq:MassConservation} implies that
$\int_{\Gamma} \bu \cdot \bn \, dS = 0$. 
It turns out that 
these conditions are satisfied when the Dirichlet 
problem is recast as a Stokes flow \cref{eq:biharm5}, as 
it is easily verified that
\begin{equation}
 \int_{\Gamma_i} \bu\cdot \bn \, dS 
= \int_{\Gamma_i} \frac{\partial f}{\partial \tau} \, dS =0 \, .
\end{equation}
Thus, any stream function $\tilde{w}$ obtained for the Stokes
flow \cref{eq:biharm5} is necessarily single-valued.

\subsection{Evaluating the stream function} \label{subsec:stream}

Given compatible boundary data for the velocity field $\bu$, 
the completed double layer representation for Stokes flow 
\cref{eq:IntRepStokes} guarantees the existence of
a solution density $\boldsymbol\mu$ 
and a corresponding stream function $\tilde{w}$. The Goursat 
function formula 
for $\tilde{w}$, see \cref{sec:stokeslayer},
is necessarily single-valued, as explained above,
but it is not immediately obvious from the formula 
that this should be true.

The difficulty in the representation of $\tilde{w}$ 
comes from the part
of the stream function corresponding to the double layer potential
 \cref{eq:wDL}. The 
second term in the expression for the double layer potential is
\begin{equation}
 v_1(z) = \mbox{Re}\left[\frac{1}{4\pi i}\int_{\Gamma}\left(
\overline{\rho\left(\xi\right)}d\xi+\rho\left(\xi\right)
\overline{d\xi}\right)\log\left(\xi-z\right)\right] \, .
\end{equation}
To compute this term, in a na\"{i}ve numerical implementation,
the question of which is the appropriate branch of the 
logarithm to use would arise at many steps.
To avoid this complication, it is possible instead to compute
$v_1$, up to a constant, as the harmonic conjugate of the
function 
\begin{equation}
 v_2 = \frac{1}{4\pi}\int_{\Gamma}\left(\overline{\rho\left(\xi\right)}
d\xi+\rho\left(\xi\right)\overline{d\xi}\right)\log
\left(\left|\xi-z\right|\right) \, . \label{eq:harmconjg}
\end{equation}

We will use this approach to evalute $v_1$ numerically.
As a result of the Cauchy-Riemann equations, the 
harmonic conjugate of $v_2$, satisfies the following
Neumann problem for the Laplace equation:
\begin{align}
 \Delta v_1 &= 0 &\quad x\in D \, ,\\
 \frac{\partial v_1}{\partial n} &= -\frac{\partial v_2}{\partial \tau} 
&\quad x\in\Gamma \, .
\end{align}
It is possible then to use standard integral equation
methods to compute $v_1$. 

Let $v_1 = S^L_\Gamma \sigma$, where $S^L_\Gamma \sigma$ 
is the single layer potential for Laplace's equation, given by
\begin{equation}
S^L_\Gamma\sigma (\x) = -\frac{1}{2\pi} \int_{\Gamma} \log 
\left|\x - \y \right|\sigma(\y)\, dS \left(\y \right) \, ,
\end{equation}
where $\sigma\in \mathcal{X}= C^{0,\alpha}\left(\Gamma\right)$, for 
some $\alpha \in (0,1)$, is an unknown density 
(see \cite{kress1999linear, guenther1988partial}).
Imposing the Neumann boundary conditions results in the 
following boundary integral equation for $\sigma$:
\begin{align}
-\frac{\partial v_2}{\partial \tau} (\x) &= \frac{1}{2} \sigma \left(\bx\right) - \frac{1}{2\pi}\oint_{\Gamma} 
\frac{\partial}{\partial n_x}  
\log \left| \x - \y \right |\sigma(\y)\, dS \left(\y \right) 
\, , \\
-\frac{\partial v_2}{\partial \tau} &= \left( \frac{1}{2}I_{\mathcal{X}} + K^L_\Gamma  \right) \sigma 
\, ,
\label{eq:BlockSystemRow2P1tmp}
\end{align}
where the operator $K^L_\Gamma$ is compact, so that the integral equation
is second kind.
For a derivation of this result, see \cite{kress1999linear}.

It is well known that the operator $\frac{1}{2}I_{\mathcal{X}} + K^L_\Gamma$ 
has a one dimensional null space. Thus,
we choose to solve the above integral equation subject to 
the constraint $\int_{\Gamma} \sigma \, dS = 0$. 
Furthermore, it is known that solving the Neumann problem subject 
to the above constraint is equivalent
to solving
\begin{align}
\left( \frac{1}{2}I_{\mathcal{X}} + K^L_\Gamma  + W_\Gamma \right) \sigma 
= -\frac{\partial v_2}{\partial \tau}
\label{eq:BlockSystemRow2P1} 
\end{align}
where $W_\Gamma\sigma = \int_{\Gamma} \sigma \, dS$. 

\subsection{Making the representation complete}
As mentioned above, the solution $\tilde{w}$ of 
the auxiliary Stokes problem \cref{eq:biharm5}
only satisfies the boundary conditions of the 
original Dirichlet problem \cref{eq:biharmD1,eq:biharmD2,eq:biharmD3}
up to a constant on each boundary component. For a simply
connected domain, this constant can be recovered from the
fact that adding an arbitrary constant to a stream function
does not change the velocity field. 
Thus, in simply connected domains, there is an equivalence in the 
solutions of \cref{eq:biharm5} and \cref{eq:biharmD1,eq:biharmD2,eq:biharmD3}.

To analyze the case of a multiply connected domain,
we first consider
radially symmetric solutions on an annulus centered at the
origin. Let $w\left(r\right)$ be a 
radially symmetric biharmonic
potential. Then $w(r)$ solves the ordinary differential
equation (ODE)
\begin{equation}
 \frac{d}{dr}r\frac{d}{dr}\frac{1}{r}\frac{d}{dr}r
\frac{dw}{dr} = 0 \, .
\end{equation}
Four linearly independent solutions of this ODE are $1$, $r^2$, 
$\log \left(r\right)$, and $r^2 \log\left(r\right)$. 
For each solution, we can compute the associated velocity field 
$\bu = \nabla^{\perp} w$. By construction, $\bu$ satisfies 
the continuity condition \cref{eq:MassConservation}. 
For the momentum equation \cref{eq:StokesFlowEq} to be 
satisfied, we need that $\Delta\bu$ is a conservative
vector field, which 
is equivalent to the condition that 
$\int_{\gamma} \Delta\bu \cdot d\bl = 0$ 
for any closed 
loop $\gamma$ in the annulus. For the first three linearly independent
solutions, $\Delta \bu=0$ so that $\Delta \bu$
is trivially a conservative vector field. 
The fourth solution, on the other
hand, has $\Delta \bu=\frac{4}{r} \hat{\theta}$.
By considering a curve $\gamma$ encircling the origin,
we see that $\Delta \bu$ is not a conservative vector
field and that any pressure for the velocity field 
associated with $r^2 \log\left(r\right)$ is not single-valued.

The function $\frac{1}{8\pi} r^2 \log\left(r\right)$ 
is the Green's function for the biharmonic equation and is the 
equivalent of a \textit{charge} for such problems. The above analysis
can be extended to show that any solution of the biharmonic equation 
with net charge 
cannot be represented as a Stokes velocity field. In simply connected 
domains, since $\Delta^2 w=0$, there can be no net biharmonic charge
in the domain. For multiply connected domains with genus $N$, 
the set of stream functions for Stokes velocity fields
misses an $N$ dimensional space of solutions, corresponding 
to biharmonic charges located in the holes of the domain.
Following this reasoning, we obtain a complete representation
for biharmonic potentials on multiply connected domains
by adding $N$ charges, one per each hole of the domain,
to the representation for $w$. The details of this approach, and
the proof that it is sufficient, is in the next section.

\subsection{The integral representation}
Following the discussion in the previous two sections,
it is now possible to present an integral representation
for the Dirichlet problem of the biharmonic equation based
on the completed double layer representation
for the Stokes problem. 
We first fix some notation. Let $D$ be a multiply connected domain,
with boundary $\Gamma = \cup_{k=0}^N \Gamma_k$, as in the previous
sections. 
For each boundary component
$\Gamma_k$, let $D_k$ be its interior and $z_k$ be a point in
$D_k$. Then, let the solution $w$ be represented in terms 
of layer potentials and biharmonic charges as
\begin{equation}
w\left(z\right) = S^w_\Gamma \rho \left(z\right)
 + D^w_\Gamma \rho \left(z\right) 
+ \text{Re} 
\left[\overline{z} \int_{\Gamma} \rho \left(\xi \right) \, dS \right]
+ c_0 + \sum_{k=1}^{N} c_k r_{k}^2 
\log \left(r_k\right)  
\, ,
\label{eq:IntRepW}
\end{equation}
where $\rho$ is an unknown density, the $c_k$ are unknown
constants, the distance from $z$ to $z_k$ is $r_k = |z-z_k|$,
and the operators 
$S^w_\Gamma$ and $D^w_\Gamma$
map complex densities to the stream functions corresponding 
to single and double layer potentials, as defined in \cref{sec:prelim}.

\begin{remark}
  As discussed in \cref{subsec:stream}, we will only
evaluate the operator $D^w_\Gamma$ up to a constant in 
our numerical implementation. However, this does not affect
the analysis of this section because of the freedom in 
choosing $c_0$.
\end{remark}

As before, we can identify a real, vector-valued density 
$\boldsymbol{\mu} = (\mu_1,\mu_2)^\intercal$
with $\rho$ by setting $\mu_2(\bx) - 
i \mu_1(\bx) = \rho(z)$. Let $\bu = \nabla^\perp w$
be the velocity field corresponding to the stream function $w$.
Then, in terms of $\boldsymbol{\mu}$, we have

\begin{equation}
\bu \left(\bx\right)= 
\mathbf{S}_{\Gamma} \boldsymbol{\mu} \left(\bx\right) + 
\mathbf{D}_{\Gamma} \boldsymbol{\mu} \left(\bx\right) +
\mathbf{W}_\Gamma \boldsymbol{\mu} + 
\nabla^{\perp} \sum_{k=1}^{N} c_k r_{k}^2 \log (r_k)
\, , 
\label{eq:IntRepU}
\end{equation}
where $\mathbf{S}_{\Gamma} \boldsymbol{\mu}$ and 
$\mathbf{D}_{\Gamma} \boldsymbol{\mu}$ are the single and double
layer potentials for the density $\boldsymbol{\mu}$, as defined
in section 2.

Let $\alpha \in \left(0,1 \right)$ 
and $\mathcal{X} = C^{0,\alpha} \left(\Gamma \right)$. 
Assume that the boundary data
for the Dirichlet problem \cref{eq:biharmD1,eq:biharmD2,eq:biharmD3} satisfies 
$f \in C^{1,\alpha} \left(\Gamma \right)$ and $g \in \mathcal{X}$,
a slightly stronger assumption on the regularity of $f$
than given in the original problem statement.
Denote the integrals of $f$ around each boundary
component by $b_k = \int_{\Gamma_k} f \, dS$. 
To solve equation 
\cref{eq:biharmD1,eq:biharmD2,eq:biharmD3}, we impose the boundary conditions on
the gradient of $w$ as in \cref{eq:biharm5} on the above 
representation for $w$,
or, equivalently, the no-slip boundary conditions
\cref{eq:bcStokesFlow} 
on the above representation for $\bu$ with
\begin{equation}
\bh = \left(-\left(\frac{\partial f}{\partial \tau}\tau_2 
+ g n_2\right), 
\frac{\partial f}{\partial \tau}\tau_1 + g n_1 \right)^\intercal \, .
\end{equation}
Under the assumptions on $f$ and $g$, the boundary data $\bh 
\in \mathcal{X} \times \mathcal{X}$.

Let $\bB\bc(\x)$ denote the part of the 
velocity field due to the charges, i.e.

\begin{equation}
  \bB\bc(\x) = 
  \nabla^{\perp} \sum_{k=1}^{N} c_k r_{k}^2 \log ( r_k) \; .
\end{equation}
Then, due to \cref{lem:slppropspart1,lem:dlppropspart1}, 
enforcing the boundary condition $\bu(\x) = \bh(\x)$ 
for each $\x \in \Gamma$ results in the following boundary 
integral equation
\begin{align}
\bh\left(\bx\right) &= -\frac{1}{2} \boldsymbol{\mu}
\left(\bx\right) + 
\mathbf{S}_{\Gamma} \boldsymbol{\mu} \left(\bx\right) +
\mathbf{D}_{\Gamma}^{PV} \boldsymbol{\mu} \left(\bx\right) +
\mathbf{W}_\Gamma \boldsymbol{\mu}(\x)
\nonumber \\
&\quad + \nabla^{\perp} \sum_{k=1}^{N} c_k r_{k}^2 
\log (r_k)
 \\
\bh(\x) &= \left(-\frac{1}{2} \mathbf{I}_{\mathcal{X}\times\mathcal{X}} 
+ \mathbf{K}_\Gamma\right) \boldsymbol{\mu}(\x) +
\bB\bc(\x)
\label{eq:BlockSystemRow1new}
\, .
\end{align}

To ensure that the values of $w$ are correct on the boundary,
further constraints are needed.
We impose $N+1$ additional conditions on the value of $w$
\begin{equation}
 \int_{\Gamma_k} w \, dS = b_k \quad k=0,1,2,\ldots , N \label{eq:IntConstNew}
\, ,
\end{equation}
where the constants $b_k$ are as defined above. The
integral of $w$ about each component can be written
in terms of the unknowns as 

\begin{align}
\int_{\Gamma_k} w\left(\bx \right) \, dS(\x) &= 
\int_{\Gamma_k} \left[
S^w_\Gamma \left[-\mu_2 + i\mu_1 \right] \left( \xi \right )
+ D^w_\Gamma \left[-\mu_2 + i\mu_1 \right] \left( \xi \right )
\right ] \, dS_{\xi}  \nonumber \\
&\quad + \int_{\Gamma_k} \left [ \alpha \int_{\Gamma} 
\boldsymbol{\mu} \left(\by \right)
\cdot \bx \, dS_{\by} \right ] \, dS(\x)
\nonumber \\
&\quad + \int_{\Gamma_k} \left[ c_0 + 
\sum_{l=1}^{N} c_l r_{l}^2 \log r_l \right] \, dS(\x) \\
& =: D_{k} \boldsymbol{\mu} + F_{k}\bc 
\label{eq:BlockSystemRow2new}
\end{align}
Combining equations \cref{eq:BlockSystemRow1new}, \cref{eq:IntConstNew}, 
and \cref{eq:BlockSystemRow2new},
we get the following linear system for the 
unknowns $\boldsymbol{\mu}$ and $\bc$
\begin{equation}
 \begin{bmatrix}
  -\frac{1}{2}\mathbf{I}_{\mathcal{X}\times\mathcal{X}}
  +\mathbf{K}_\Gamma &  \bB \\
    \mathbf{D} & \mathbf{F}
 \end{bmatrix}
\begin{bmatrix}
 \boldsymbol{\mu} \\
 \bc
\end{bmatrix} = 
\begin{bmatrix}
 \bh\\
\bb 
\end{bmatrix} 
\label{eq:IntEqBiharm1new} \; ,
\end{equation}
where $\mathbf{D} = \left(D_0, \ldots D_N\right)^\intercal$, 
$\mathbf{F} = \left(F_0, \ldots F_N\right)^\intercal$, 
and $\bB = \left(b_0, \ldots b_N\right)^\intercal$. 

\begin{proposition}
The block system \cref{eq:IntEqBiharm1new}
is an invertible Fredholm operator.
\end{proposition}

\begin{proof}
It is simple to
show that the linear system \cref{eq:IntEqBiharm1new}
is Fredholm. The block which contains
$-1/2 \mathbf{I}_{\mathcal{X}\times\mathcal{X}}
+ \mathbf{K}_\Gamma$ is Fredholm due to 
\cref{lem:stokesrep}. The off-diagonal blocks,
denoted by $\bB$ and $\mathbf{D}$, 
are trivially compact because either the domain 
or range of the operator is finite dimensional. 
Finally, $\mathbf{F}$ is 
Fredholm because it is a finite-dimensional linear
operator. Therefore, the full system is Fredholm.

Due to the Fredholm alternative,
it is only necessary to establish the injectivity of
the system \cref{eq:IntEqBiharm1new} to prove that 
it is invertible.
It is clear that if $\boldsymbol{\mu}$ and $\bc$ 
solve equation \cref{eq:IntEqBiharm1new}, then the 
resulting solution, $w$, given by \cref{eq:IntRepW}, 
solves the original Dirichlet
problem \cref{eq:biharmD1,eq:biharmD2,eq:biharmD3}. By construction, 
$w$ is biharmonic in $D$. Moreover, $w$ satisfies 
$\frac{\partial w}{\partial \tau} = \frac{\partial f}{\partial \tau}$
and $\frac{\partial w}{\partial n} = g$ on the whole boundary
$\Gamma$ and $\int_{\Gamma_{k}} w = \int_{\Gamma_k} f$ for each
boundary component $\Gamma_k$, so that the boundary conditions are
satisfied. 

In the case that $\bh \equiv \bzero$ and 
$\bb = \bzero$,
we have that $f = g \equiv 0$ for the Dirichlet problem.
By the uniqueness of solutions to \cref{eq:biharmD1,eq:biharmD2,eq:biharmD3}, this implies
that $w \equiv 0$ in $D$. It is, however, less immediate 
that $w \equiv 0$ implies that $\boldsymbol{\mu} \equiv 
\bzero$ and $\bc = \bzero$. 

For each $k = 1,\ldots,N$, 
let $\tilde{\Gamma}_{k} \subset D$ be a curve which satisfies 
$n\left(z_{j}, \tilde{\Gamma}_{k}\right) = \delta_{jk}$, 
where $n\left(z,\gamma\right)$ represents the winding number 
of the curve $\gamma$ about z. 
Because $\bu = \nabla^{\perp} w$ and $w\equiv 0$ in $D$,
we have 
\begin{equation}
\int_{\tilde{\Gamma}_k} \Delta \bu \cdot \boldsymbol{\tau} \, dS = 0 .
\end{equation}

Let $\bu^\mu = \mathbf{S}_\Gamma \boldsymbol{\mu}
+ \mathbf{D}_\Gamma \boldsymbol{\mu} = \bu - \bB\bc 
- \mathbf{W}_\Gamma \boldsymbol{\mu}$. 
We observe that $\bu^\mu$ corresponds
to a Stokes velocity field in $D$ for any $\boldsymbol{\mu}$. 
Let $p$ be its associated pressure. Then
\begin{equation}
\int_{\tilde{\Gamma}_k} \Delta \bu^\mu 
\cdot \boldsymbol{\tau}
\, dS = \int_{\tilde{\Gamma}_k} \nabla p \cdot \boldsymbol{\tau} \, dS = 0 \, .
\end{equation}
Further, a simple calculation shows that 
\begin{equation}
\int_{\tilde{\Gamma}_k} \Delta \nabla^{\perp} c_j 
r_j^2 \log (r_j) \cdot \boldsymbol{\tau} \, dS = 
8 \pi c_{j} \delta_{jk} \, ,
\end{equation}
for $j = 1, \ldots, N$.
Combining these equations, we conclude that
\begin{equation}
0 = \int_{\tilde{\Gamma}_k} \Delta \bu \cdot 
\boldsymbol{\tau} \, dS =
\int_{\tilde{\Gamma}_k} \Delta ( \bu^\mu 
+ {\mathbf W} \boldsymbol{\mu} + \bB\bc )
\cdot \boldsymbol{\tau} \, dS =  
8\pi c_k \, .
\end{equation}
Thus $c_k = 0$ for $k=1,2,\ldots N$. 

The first row of the system \cref{eq:IntEqBiharm1new} then reads
\begin{equation}
\left(-\frac{1}{2}\mathbf{I}_{\mathcal{X}\times\mathcal{X}} 
+ \mathbf{K}_\Gamma \right) 
\boldsymbol{\mu} = 0 \, . \label{eq:IntEq1}
\end{equation}
From the invertibility of 
$-\frac{1}{2}\mathbf{I}_{\mathcal{X}\times\mathcal{X}} 
+ \mathbf{K}_\Gamma$, we
conclude that $\boldsymbol{\mu} \equiv \bzero$. 
Because $\boldsymbol{\mu}\equiv \bzero$ and 
$c_k = 0$ for $k = 1, \ldots , N$,
we get that $w \equiv c_0$. It then follows that $c_0 = 0$
as well, proving the injectivity of the system. 
\end{proof}

\section{Results} \label{sec:results}

We first review the existing numerical tools used to compute
solutions of the integral equation~\cref{eq:IntEqBiharm1new}.
To discretize the integral equations, we use the
Nystr\"{o}m method. 
We divide the boundary into
panels and represent the unknown density and the
boundary data by their values at scaled Gauss-Legendre
nodes on each panel. 
Let $n_{p}$ denote the number of Gauss-Legendre panels.
We discretize each panel using 16 scaled Gauss-Legendre nodes.
Then $n_{d} = 16 n_{p}$ is the number of discretization points 
on the boundary.
Let $\bx_{j}$ denote the discretization nodes, $w_{j}$ denote the
appropriately scaled Gauss-Legendre quadrature weights for smooth functions,
and $\bmu_{j}$ denote the unknown density at $\bx_{j}$.
When forming the linear system, we use scaled unknowns, 
$\bmu_{j} \sqrt{w_{j}}$, so that the spectral properties of the 
discrete system with respect to the $l_2$ norm
are approximations of the spectral properties of the continuous
system as on operator on $L_2$ (for more on this point of view, 
see \cite{bremer2012}).
The integral kernels in this paper are either smooth or have 
a weak (logarithmic) singularity. 
For the smooth kernels in the integral representation, we use standard 
Gauss-Legendre weights appropriately scaled.
For kernels with a logarithmic singularity, we use order $20$ 
generalized Gaussian quadrature rules~\cite{bremer2010,bremer2010u}.

After applying the integral rule, we obtain
a linear system for the unknowns. This system is
typically well-conditioned, but dense. 
Let $\bA$ denote the discretized linear system of size $2n_{d}+N+1$
corresponding to the integral equation~\eqref{eq:IntEqBiharm1new}.
Let $\kappa(\bA)$ denote the condition number of the discretized matrix
$\bA$. 
For our applications, the system size was modest and 
we computed the unknowns $\boldsymbol{\mu}$ and $c_{k}$ 
using Gaussian elimination. 
For larger applications,
the system is amenable to solution by any of a
variety of iterative or fast-direct solvers,
which we will not review here.

For the visualizations in this section, we
evaluate the layer potentials inside the domain,
with some points being very close to the boundary.
The value of the layer potential can be difficult
to evaluate at such points because of the
near-singularity in the integral kernel. We
use a sixth order quadrature by expansion method
\cite{klockner2013,rachh2016} to evaluate these
integrals efficiently and accurately.

In this section we consider two test cases.
The first example is a convergence study for a simply
connected domain to demonstrate the order
of convergence for the discretized integral equation.
We also compare the condition numbers 
for the discretized linear systems corresponding to 
our integral representation
and the existing integral representation by Farkas~\cite{Farkas89}
for a family of simply connected domains with increasing curvature.
For the second example, we demonstrate the order of convergence and
compute the Green's function for a multiply connected domain.
For all examples except the computation of the Green's function,
the boundary data $f$ and $g$ are chosen corresponding to 
a known solution of the biharmonic equation in $D$ given by

\begin{equation} \label{eq:sourceformula}
w(\bx) = \sum_{j=1}^{n_{s}} q_{j} |\bx - \bs_{j}|^2 \log{|\bx-\bs_{j}|} \, ,
\end{equation}
where $q_{j}$ are uniformly chosen from $[0,1]$.
Let $\wcomp(\bt)$ denote the computed solution at targets 
$\bt$ in the interior of $D$, and let $\eps$ denote an estimate for the error
given by
$$
\eps = \frac{
\sqrt{\sum_{j=1}^{n_{t}} (\wcomp(\bt_{j}) - w(\bt_{j}))^2} }{
\sqrt{\sum_{j=1}^{n_{t}}
w(\bt_{j})^2} } \, .
$$

\subsection{Simply connected domain examples}
Let $D$ denote the interior of a rounded rectangular
bar with length $a=1$, height $b=0.5$, and vertices at
$(0,0),(a,0),(a,b),(0,b)$.
Following the procedure discussed in~\cite{mikecornerrounding},
the corners are rounded using the Gaussian kernel
$$
\phi(x) = \frac{1}{\sqrt{2\pi h}} e^{-x^2/(2 h^2)} \, ,
$$
with $h=0.05$. 
The boundary data $f$ and $g$ are chosen corresponding to a
known solution $w$, defined as in~\cref{eq:sourceformula}, 
with four sources $\bs_{j}$ located at
\begin{align*}
\bs_{1} = \left(a+0.2+ \delta_{1},\frac{b}{2} + \delta_{2}\right) \, , \, \,
&\bs_{2} = \left(\frac{a}{2}+\delta_{3},b+0.2+\delta_{4}\right) \, , \, \, \\
\bs_{3} = \left(-0.2+\delta_{5},\frac{b}{2} + \delta_{6}\right) \, , \, \,
&\bs_{4} = \left(\frac{a}{2} + \delta_{7},-0.2 + \delta_{8}\right) \, ,
\end{align*}
with $\delta_{i}$ chosen uniformly from $[-0.05,0.05]$.

The potential $w$ is evaluated at targets $\bt_{j}$ 
in the interior of $D$,
\begin{align*}
\bt_{1} = \left( \frac{a}{4}+\delta_{9}, \frac{b}{4}+\delta_{10} \right) \, , \, \,
&\bt_{2} = \left( \frac{a}{4}+\delta_{11}, \frac{3b}{4}+\delta_{12} \right) \, , \, \, \\
\bt_{3} = \left( \frac{3a}{4}+\delta_{13}, \frac{b}{4}+\delta_{14} \right) \, , \, \,
&\bt_{4} = \left( \frac{3a}{4}+\delta_{15}, \frac{3b}{4}+\delta_{16} \right) \, , 
\end{align*}
with $\delta_{i}$ again chosen uniformly from $[-0.05,0.05]$.
A sample geometry with sources $\bs_{j}$ and targets $\bt_{j}$ 
and the error $\eps$ as a function of $n_{d}$ are shown
in \cref{fig:simplyconnectedfig}.
The convergence study shows that the error decays like a $20$th order
convergent scheme. 
\begin{figure}
\begin{center}
\resizebox{6cm}{!}
{
\input{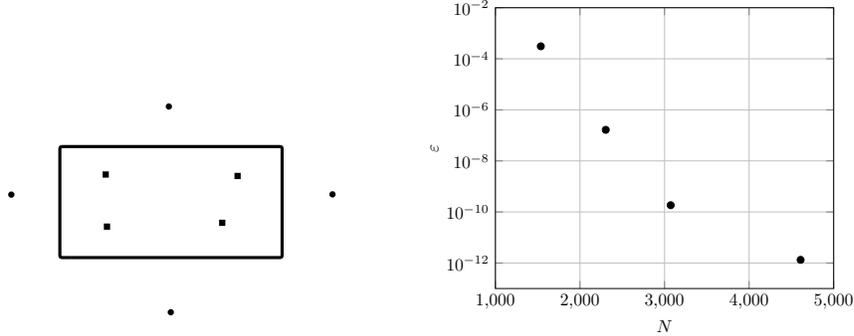}
}
\resizebox{6cm}{!}
{
\begin{tikzpicture}
\begin{axis}[
xmin=1000, xmax=5000, ymin = 1.0E-13, ymax = 1.0E-2,
xlabel={$N$}, ylabel={$\eps$}, 
xmode=linear,ymode=log,grid=both,
]
\addplot[scatter, only marks,
scatter/classes=
{a={mark=*,blue}}
]
table[col sep=comma, meta=label]{
x,y,label
1537,3.09e-4,a
2305,1.65e-7,a
3073,1.84e-10,a
4609,1.33e-12,a
    };
\end{axis}
\end{tikzpicture}
}
\end{center}
\caption{(left): Geometry of simply connected domain for convergence study -- 
the circles denote the location
of the sources $\{\bs_{j}\}$ and the squares denote the location of the
targets $\{\bt_{j} \}$, (right): error $\eps$ as a function of the system
size $N=2 n_{d} +1$.
}
\label{fig:simplyconnectedfig}
\end{figure}

The integral equation presented in this paper is significantly
better conditioned than the existing integral equation discussed
in~\cite{Farkas89} --- particularly when the boundary has regions
with large curvature. 
We plot the condition number $\kappa(\bA)$ of the discretized system 
of integral equations for the representations given by 
both~\cref{eq:IntEqBiharm1new} 
and~\cref{eq:farkIntEq} as a function of the rounding 
parameter $h$ for the rounded rectangular bar in~\cref{fig:condcomp}.
The maximum curvature of the boundary is directly proportional to
$1/h^2$.
The condition number $\kappa(\bA)$
increases linearly with the maximum curvature
for integral equation~\cref{eq:farkIntEq}, but is independent
of the curvature for the integral equation presented in this paper.
\begin{figure}
\begin{center}
\resizebox{8cm}{!}
{
\begin{tikzpicture}
\begin{axis}[
xmin=0.02, xmax=0.3, ymin = 1.0E1, ymax = 2.0E7,
xlabel={$h$}, ylabel={$\kappa(\bA)$}, 
xmode=log,ymode=log,grid=both,scatter/classes=
{a={mark=*,black},
b={mark=square*,black}},
xtick={0.025,0.05,0.1,0.2},
xticklabels={$0.025$,$0.05$,$0.1$,$0.2$}
]
\addplot[scatter, only marks,scatter src=explicit symbolic
]
table[col sep=comma,meta=label]{
x,y,label
0.2,2.22e1,a
0.1,2.47e1,a
0.05,2.71e1,a
0.025,2.95e1,a
0.2,9.59e3,b
0.1,4.05e4,b
0.05,1.83e6,b
0.025,8.79e6,b
  
    };
\end{axis}
\end{tikzpicture}
}
\end{center}
\caption{Condition number $\kappa(\bA)$ of discretized integral equations
~\cref{eq:IntEqBiharm1new} (circles) and~\cref{eq:farkIntEq} (squares)
as a function of corner rounding parameter $h$}
\label{fig:condcomp}
\end{figure}
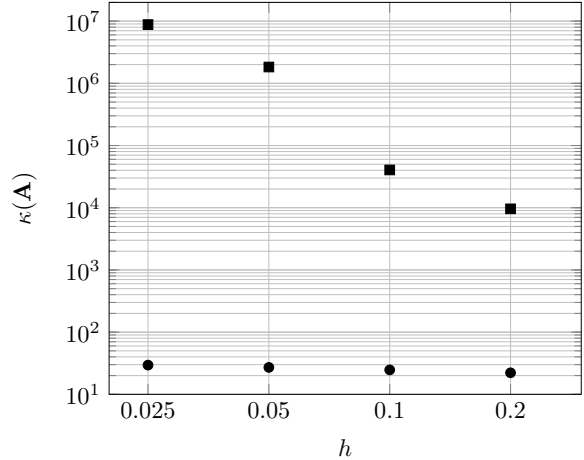

\subsection{Multiply connected domain - examples}
Let $D$ now denote the interior of a multiply connected domain,
where the outer boundary $\Gamma_{0}$ is the boundary
of the rounded rectangular bar discussed above with rounding parameter
$h=0.05$ and the domain has ten circular obstacles $\Gamma_{i}$ 
with radii $r_{0} = 0.04$ and centers located at $\bx_{i}$,
\begin{align*}
\bx_{i} &= \left( 0.12 + (i-1)0.2, 0.15 \right) \quad i=1,2,\ldots 5 \, , \\
\bx_{i} &= \left( 0.08 + (i-6)0.2, 0.35 \right) \quad i=6,7,\ldots 10 \, .
\end{align*}

We will first perform a convergence study, as above,
with a known solution $w$ defined in terms of point
sources according to~\cref{eq:sourceformula}. We create ten sources, 
one located inside each obstacle, whose locations are given by
$$
\bs_{i} = \bx_{i} + (\delta_{2i-1},\delta_{2i}) \, ,
$$
where $\delta_{i}$ are chosen uniformly from $[-0.5r_0,0.5r_{0}]$.
The potential is then tested at twelve targets located at
\begin{align}
\bt_{i} &= \left( 0.22+(i-1)0.2,0.05 \right) + (\delta_{2i-1},\delta_{2i}) 
\quad i=1,2,3,4 \, , \nonumber\\
\bt_{i} &= \left( 0.22+(i-5)0.2,0.25 \right) + (\delta_{2i-1},\delta_{2i}) 
\quad i=5,6,7,8 \, , \label{eq:targmultloc}\\
\bt_{i} &= \left( 0.18+(i-9)0.2,0.45 \right) + (\delta_{2i-1},\delta_{2i}) 
\quad i=9,10,11,12 \, , \nonumber
\end{align}
where $\delta_{i}$ are chosen uniformly from $[-0.5r_{0},0.5r_{0}]$.

We observe $20$th order convergence in the error even for this example.
The error as a function of the number of discretization points along
with a sample geometry are shown in~\cref{fig:mult-connected}.
We also plot the field, and the error in evaluating the potential
in the volume using a sixth order quadrature by expansion method
in~\cref{fig:pot-err}. We note that the error observed near the boundary
is larger than at the targets used for the convergence study; this
is a result of the relatively low order of the quadrature by expansion
method and could be improved by increasing the number of points 
on the boundary.
\begin{figure}
\begin{center}
\resizebox{6cm}{!}
{
\input{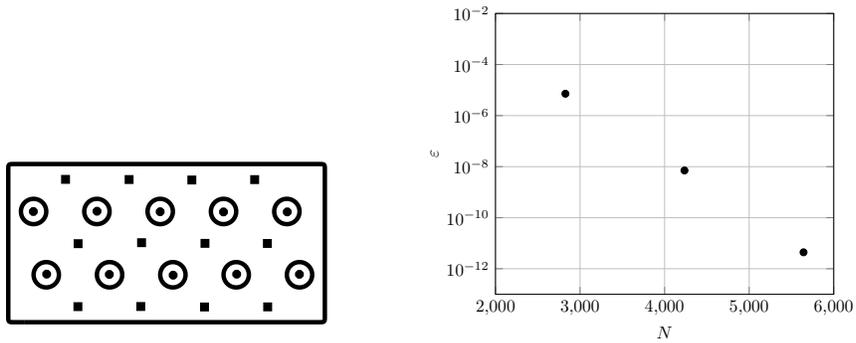}
}
\resizebox{6cm}{!}
{
\begin{tikzpicture}
\begin{axis}[
xmin=2000, xmax=6000, ymin = 1.0E-13, ymax = 1.0E-2,
xlabel={$N$}, ylabel={$\eps$}, 
xmode=linear,ymode=log,grid=both,
]
\addplot[scatter, only marks,
scatter/classes=
{a={mark=*,blue}}
]
table[col sep=comma, meta=label]{
x,y,label
2827,7.19e-6,a
4235,7.08e-9,a
5643,4.41e-12,a
    };
\end{axis}
\end{tikzpicture}
}
\end{center}
\caption{(left): Geometry of multiply connected domain 
for convergence study -- the circles denote the location
of the sources $\{\bs_{j}\}$ and the squares denote the location of the
targets $\{\bt_{j} \}$, (right): error $\eps$ as a function of the system
size $N=2 n_{d} +N+1$.
}
\label{fig:mult-connected}
\end{figure}

\begin{figure}
\begin{center}
\includegraphics[width=6cm]{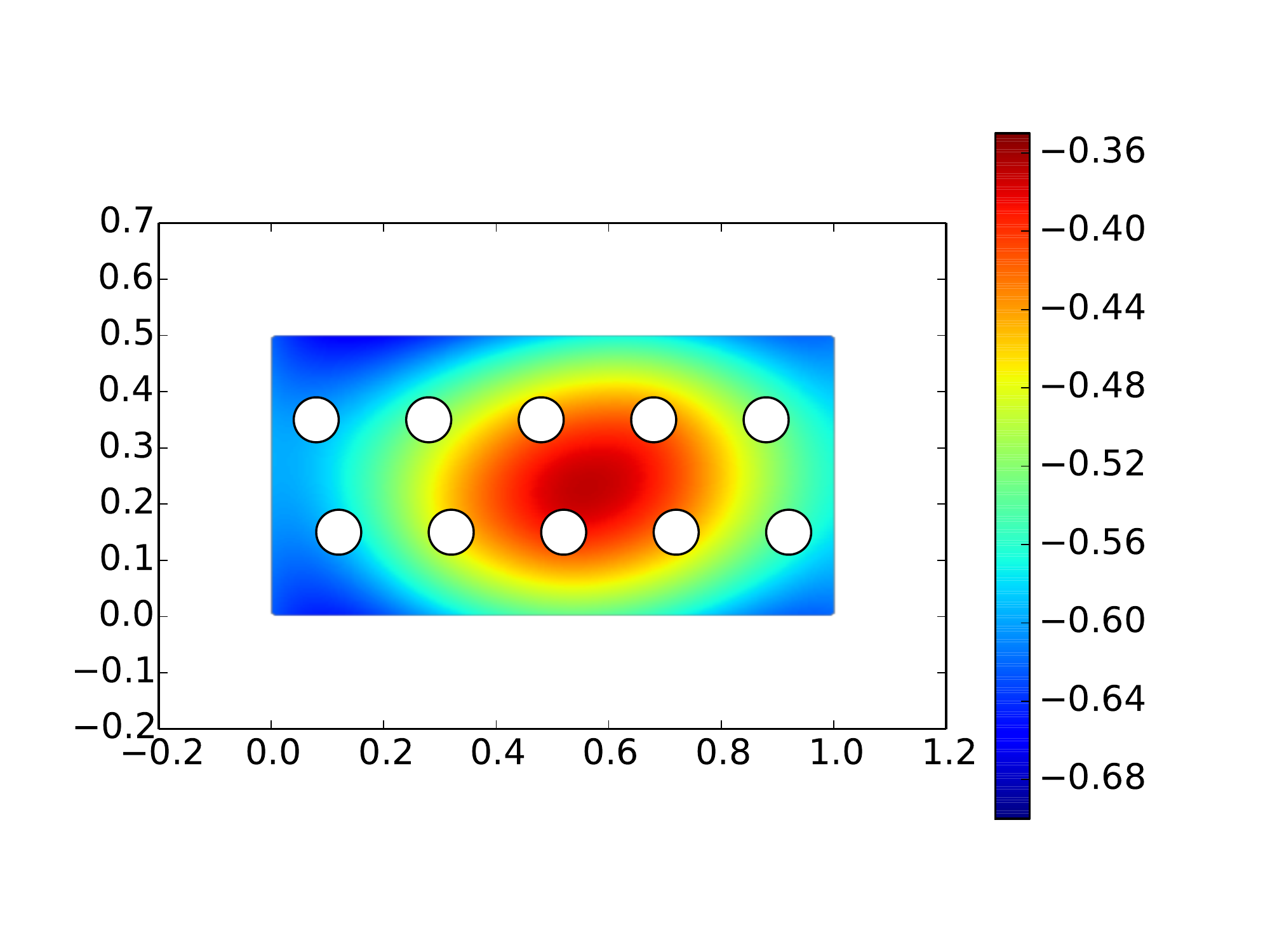}
\includegraphics[width=6cm]{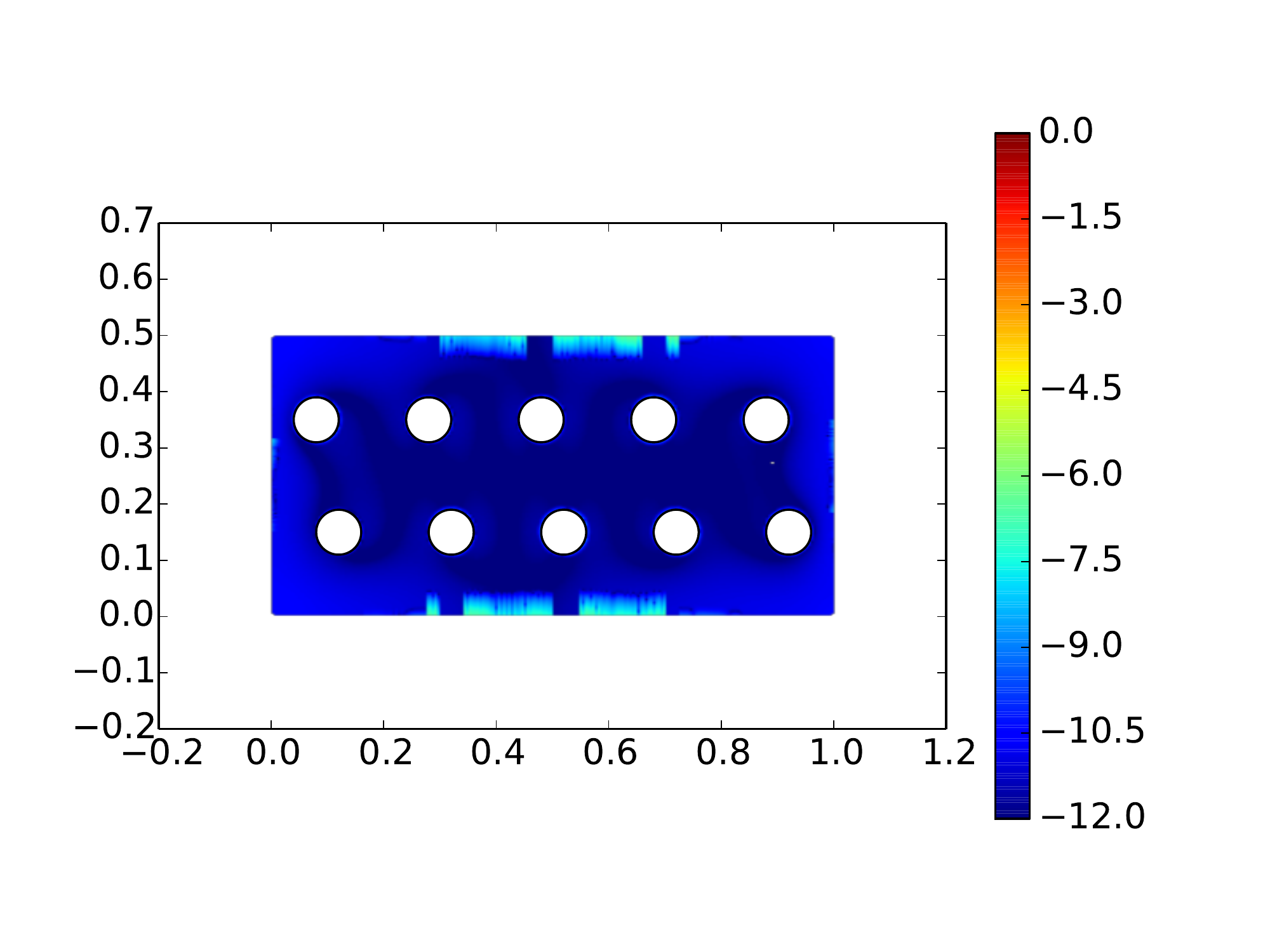}
\end{center}
\caption{(left): Known biharmonic potential due to sources
at $\{\bs_{j}\}$ and (right) absolute pointwise error
$|\wcomp(\bt) - w(\bt)|$ for targets in the interior of 
$D$.  
}
\label{fig:pot-err}
\end{figure}

For the final example, we compute the function $w$ 
which satisfies the PDE
\begin{align}
 -\Delta^2 w = \sum_{j=1}^{12} \delta_{\bx=\bt_{j}} &\quad \bx \in D  
\nonumber\\
 w = 0 &\quad \bx \in \Gamma \label{eq:domgreenfun}\\ 
 \frac{\partial w}{\partial n} =0 &\quad \bx \in \Gamma \nonumber  \, ,
\end{align}
where $\delta_{\bx=\by}$ is the two dimensional radially symmetric 
Dirac delta function centered at $\by$ and 
$\{\bt_{j}\}$ are defined in~\cref{eq:targmultloc}.
This function describes the vertical displacement 
of an isotropic and homogeneous thin clamped plate with
a transverse load given by point 
forces at the points $\bt_j$. It is also, by definition,
a linear combination of the domain Green's function $G^D$, 
as in
$$
w(\bx) = \sum_{j=1}^{12} G^{D}(\bx,\bt_{j}) \, .
$$

To compute $w$, we first obtain a particular 
solution $w_{p}$ which satisfies the PDE in the volume and add
to it the solution of a homogeneous problem $w_{h}$ to fix
the boundary conditions. We have
$$
w(\bx) = w_{p}(\bx) + w_{h}(\bx) \, ,
$$
where
$$
w_{p}(\bx) = \sum_{j=1}^{12} G^{B} (\bx,\bt_{j}) \, ,
$$
and $w_{h}$ satisfies the following homogeneous biharmonic equation,
\begin{align}
 -\Delta^2 w_{h} = 0 &\quad \bx \in D 
\nonumber\\
 w_{h} = -w_{p} &\quad \bx \in \Gamma \\ 
 \frac{\partial w_{h}}{\partial n} = -\frac{\partial w_{p}}{\partial n} &\quad \bx \in \Gamma \nonumber  \,.
\end{align}
We plot the computed solution in~\cref{fig:greenfunplot}.
\begin{figure}
\begin{center}
\includegraphics[width=.8\textwidth]{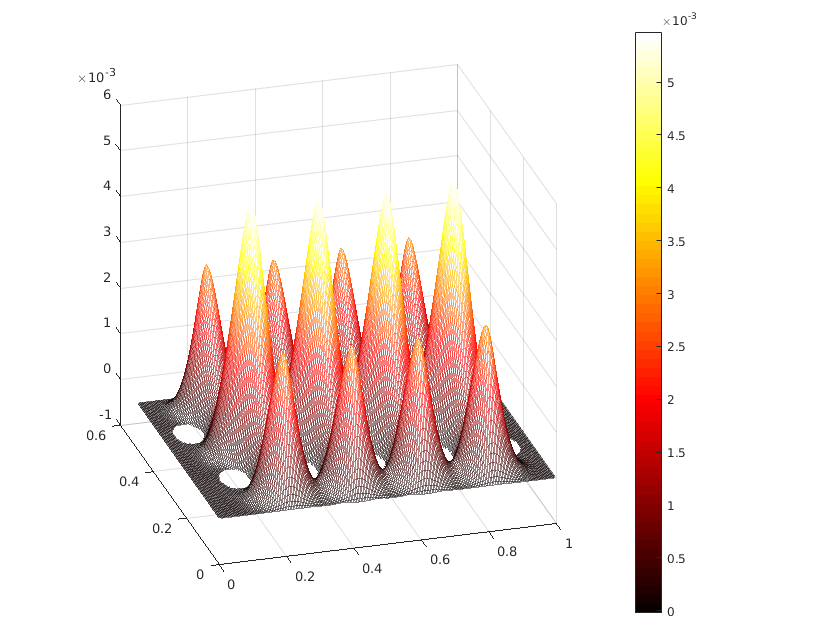}
\end{center}
\caption{Biharmonic domain green's function satisfying equations~\cref{eq:domgreenfun}}
\label{fig:greenfunplot}
\end{figure}

\section{Conclusion} \label{sec:conclusion}

We have presented an integral representation
for the biharmonic Dirichlet problem which
is stable for domains which have a boundary
with high curvature and is applicable to
domains which are multiply connected. The
representation is based on converting the
Dirichlet problem into a problem with
velocity boundary conditions, so that classical
representations for the velocity boundary
value problem can be used. While the
technique of \cite{Farkas89} --- in which
integral kernels are chosen by optimizing over
the derivatives of an appropriate Green's
function --- is general and powerful, the
spectral properties of the resulting operator
are undesirable for boundaries with high curvature
or a corner. Indeed, it seems intuitive that
all direct representations for the biharmonic
Dirichlet problem should suffer in some way:
such an approach asks that one of the integral
kernels be singular enough to result in a
second kind Fredholm equation for the value of the
layer potential and smooth enough to result
in a first kind Fredholm equation for the
normal derivative of the layer potential.

While some of the above is specific
to the biharmonic equation, in particular
the use of Goursat functions, it is
reasonable to expect the
approach to generalize to other high
order elliptic problems as well. In particular,
there are representations for the modified
Stokes equations which are analogues of the
completed double layer representation
used here \cite{pozrikidis1992boundary}.
The extension of this method to three dimensions
is a topic of ongoing research and will be
reported at a later date.

\section*{Acknowledgments}
M.~Rachh's work was supported by 
the U.S. Department of Energy under contract DEFG0288ER25053, the Office
of the Assistant Secretary of Defense for Research and Engineering and
AFOSR under NSSEFF Program Award FA9550-10-1-0180, and
the Office of Naval Research under award N00014-14-1-0797/16-1-2123.
T.~Askham's work was supported by
the U.S. Department of Energy under contract DEFG0288ER25053, 
the Office
of the Assistant Secretary of Defense for Research and Engineering and
AFOSR under NSSEFF Program Award FA9550-10-1-0180, and
the Office
of the Assistant Secretary of Defense for Research and Engineering and
AFOSR under award FA95550-15-1-0385.
The authors would like to thank Leslie Greengard for
suggesting this problem and both Leslie Greengard and
Shidong Jiang for many useful discussions.

%
%

\bibliography{dirichlet-biharm-2017}

\bibliographystyle{elsarticle-num}

\nocite{*}

%
%

\end{document}